\documentclass[reqno]{amsart}
\usepackage{geometry}
\usepackage{amssymb}
\usepackage{amsmath,amsthm}
\usepackage{amsmath,amscd}
\usepackage{amsfonts}
\usepackage{pstricks}
\usepackage{epsf}
\usepackage{xypic}
\usepackage{graphicx}
\usepackage{extarrows}
\usepackage{chemarrow}
\usepackage{xcolor}
\usepackage[colorlinks=true,linkcolor=blue,citecolor=blue]{hyperref}
\usepackage{amsmath}
\usepackage[noadjust]{cite}
\newtheorem{theorem}{Theorem}[section]
\newtheorem{lemma}[theorem]{Lemma}
\newtheorem{corollary}[theorem]{Corollary}
\newtheorem{proposition}[theorem]{Proposition}
\theoremstyle{definition}
\newtheorem{definition}[theorem]{Definition}
\newtheorem{example}[theorem]{Example}

\theoremstyle{remark}

\numberwithin{equation}{section}

\DeclareSymbolFont{EulerExtension}{U}{euex}{m}{n}
\DeclareMathSymbol{\euintop}{\mathop} {EulerExtension}{"52}
\DeclareMathSymbol{\euointop}{\mathop} {EulerExtension}{"48}

\def\dr{\ar@{->}[r]}

\def\Add{\mathsf{Add}}
\def\Prod{\mathsf{Prod}\hspace{.01in}}
\def\Ext{\mathsf{Ext}}

\def\Hom{\mbox{Hom}}

\def\cg{\mathsf{Cogen}\hspace{.01in}}
\def\gen{\mathsf{Gen}\hspace{.01in}}

\def\End{\mathsf{End}}
\def\Ker{\mathsf{Ker}\hspace{.01in}}
\def\Im{\mathsf{Im}\hspace{.01in}}

\newcommand{\ann}{\mathsf{ann}}
\newcommand{\gl}{\mathsf{gl.dim}\hspace{.01in}}

\def\mod{\mathsf{mod}\hspace{.01in}}
\def\Mod{\mathsf{Mod}\hspace{.01in}}

\def\Tr{\mathsf{Tr}}
\def\rej{\mathsf{Rej}}

\begin{document}

\title{Cosilting modules arising from cotilting objects}

\author{Yonggang Hu}
\address{College of Applied Sciences, Beijing University of Technology, 100124 Beijing, P. R. China. }
\email{huyonggang@emails.bjut.edu.cn}
\thanks{Yonggang Hu  was supported by National Natural Science Foundation of China (Grant  Nos. 11671126, 12071120). Panyue Zhou was supported
	by the National Natural Science Foundation of China (Grant No. 11901190) and by the Scientific Research Fund of Hunan Provincial Education Department (Grant No. 19B239).}

\author{Panyue Zhou}
\address{College of Mathematics, Hunan Institute of Science and
Technology, 414006 Yueyang, Hunan,
	P. R. China.}
\email{panyuezhou@163.com}
\thanks{}

\subjclass[2020]{Primary 16D90, 18E10; Secondary 16D10, 18G15.}



\keywords{Cosilting modules, Cotilting objects, Subgenerated categories.}

\begin{abstract}
Let $R$ be a ring.  In this paper, we study the characterization of cosilting modules and establish a relation between cosilting modules and cotilting objects in a Grothendieck category. We proved that each cosilting right $R$-module $T$ can be described as a cotilting object in $\sigma[R/I]$, where $I$ is a right ideal of $R$ determined by  $T$ and $\sigma[R/I]$ is the full
subcategory of right $R$-modules, consisting of submodules of $R/I$-generated modules. Conversely, under some suitable conditions, if $T$ is a cotilting object in $\sigma[R/I]$, then $T$ is cosilting.
\end{abstract}

\maketitle

\section{Introduction}
In order to study $t$-structures in the bounded derived category of representations of Dynkin quivers, Keller and Vossieck in \cite{Keller} introduced the notion of silting objects in triangulated categories. In \cite{Koenig}, Koenig and Yang established bijective correspondences between other important concepts such as simple-minded collections,  bounded $t$-structures with length heart and  bounded co-$t$-structures. So, silting theory plays an important role in the study of homotopy or derived categories.\par
 Later, Angeleri-H\"{u}gel,  Marks, and  Vit\'{o}ria
\cite{Angeleri}  introduced the notion of  silting module, which generalizes the notion of tilting module over an arbitrary ring as well as the notion of support $\tau$-tilting module over a finite dimensional algebra (introduced by Adachi, etc. see \cite{Adachi}). Furthermore, they proved that there are bijections between silting
modules, 2-term silting complexes, certain $t$-structures
and co-$t$-structures in the derived module category. And all silting
modules are quasi-tilting modules, see \cite{Angeleri}.\par
Recently, in \cite{Breaz,Zhang}, Breaz, Pop, Zhang and Wei introduced the dual notion of silting modules, called cosilting modules. Zhang and Wei  \cite{Zhang} proved that the three concepts AIR-cotilting modules, cosilting modules and quasi-cotilting modules (\cite{Zhang02}) coincide with each other. Moreover, Breaz and Pop \cite{Breaz} gave several characterizations of  cosilting modules. In particular,  a finitely generated module $M$ over an artin algebra is silting if and only if $M^{d}$ is cosilting, where $(-)^d$ is the standard duality. However, in the general case, the above statement is not valid and  although $N$ is cosilting, $N^{d}$ is not necessarily silting, see the counterexamples \cite[Example 3.10 and 3.11]{Breaz}.\par
Let $R$ be a ring. From \cite{Wisbauer}, Wisbauer introduced the notion of $\sigma[M]$ the full subcategory of $\Mod R$ formed by the
modules subgenerated by $M$. It was shown that a good
homology theory is possible in $\sigma[M]$. For example, it is a Grothendieck category with enough injective modules and arbitrary products exists. \par
In this short paper, we study the cotilting theory in $\sigma[M]$. Inspired by the framework of Breaz and Pop in \cite{Breaz}, we  describe the connection between  cosilting modules and cotilting objects in the subgenerated category $\sigma[M]$. Now, we present our main result as follows.

\begin{theorem}
Let $R$ be a ring and $\zeta:Q_{0}\rightarrow Q_{1}$ is an $R$-homomorphism between injective modules with $T=\Ker \zeta$. Then $T$ is a cosliting module in $\Mod R$ if and only if there exists a right ideal $I$ of  $R$ such that  $T$ and $I$ satisfy the following conditions:
  \begin{enumerate}
    \item $T$ is a  cotilting object in $\sigma[R/I]$ with the injective coresolving
\begin{equation*}
0\longrightarrow  T\longrightarrow   \Tr_{R/I}(Q_{0})\xrightarrow{\Tr_{R/I}(\zeta)}\Tr_{R/I}(Q_{1})\longrightarrow  0.
\end{equation*}
    \item $\Ext^{1}_{\sigma[R/I]}(R/I,T)=0$.
    \item  {\rm$I\in\Ker\Hom_{R}(-,\mathcal{B}_{\zeta})$}.
  \end{enumerate}
\end{theorem}
In particular, if $T$ is a cosilting module over a commutative ring $R$  or $T$ is a finitely generated cosilting module over an artin algebra $R$, then $T$ is a cotilting module over a factor algebra of $R$.

The paper is organized as follows. In Section 2, we recall some basic notions and results. In Section 3, we proved our main result and give an example to explain it.\par
For some unexplain notions, the readers refer to the references \cite{Breaz,Anderson,Wisbauer}.

\section{Preliminaries}
Let $R$ be a ring with unit element. Denote by $\Mod R$ the category of right $R$-modules. If $\zeta:Q_{0}\rightarrow Q_{1}$ is an $R$-homomorphism, then class $\mathcal{B}_{\zeta}$ is defined as
$$\mathcal{B}_{\zeta}=\{X\in \Mod R~|~\Hom_{R}(X,\zeta)~\textrm{is an epimorphism}\}.$$

\begin{definition}\cite[Definition 3.1]{Breaz}
We say that an $R$-module $T$ is:
\begin{enumerate}
  \item partical cosilting (with respect to $\zeta$), if there exists an injective copresentation of $T$
      $$0\rightarrow T\xrightarrow{f}Q_{0}\xrightarrow{\zeta} Q_{1}$$
      such that:
      \begin{enumerate}
        \item $T\in\mathcal{B}_{\zeta}$, and
        \item the class $\mathcal{B}_{\zeta}$ is closed under direct products;
      \end{enumerate}
  \item cosilting (with respect to $\zeta$), if there exists an injective copresentation
      $$0\rightarrow T\xrightarrow{f}Q_{0}\xrightarrow{\zeta} Q_{1}$$
      of $T$ such that $\cg T=\mathcal{B}_{\zeta}$.
\end{enumerate}
\end{definition}

From \cite[Example 3.3]{Breaz}, we can see that every (partial) cotilting module is (partial)
cosilting and, for every ring $R$, the trivial module $0$ is cosilting. Moreover, even for some
hereditary rings, there exist non-zero cosilting modules which are not cotilting.

We collect some facts which will be used in the sequel.
\begin{proposition}{\rm\cite[Corollary 3.5 and Lemma 3.4]{Breaz}}\label{lemma2-1} Let $\zeta:Q_{0}\rightarrow Q_{1}$ is an $R$-homomorphism between injective modules with $T=\Ker \zeta$. If $T$ is a partial cosilting module, then the pair {\rm($\Ker\Hom_{R}(-,T)$, $\cg T$)} is a torison pair and {\rm$\cg T\subseteq \mathcal{B}_{\zeta}\subseteq {^{\perp}}T$}.
\end{proposition}
\begin{proposition}\label{prop2-2} Let $T$ be an $R$-module and let $E$ be an injective cogenerator in $\Mod R$. Assume that
 $0\rightarrow T\rightarrow Q_{0}\xrightarrow{\zeta} Q_{1}$ is an injective copresentation for $T$. If there exists an exact sequence
     $$0\rightarrow T_{1}\rightarrow T_{0}\xrightarrow{\xi} E$$
such that $T_{i}\in \Prod T$,  and for any $T'\in\mathcal{B}_{\zeta}$ the homomorphism {\rm$\Hom_{R}(T',\xi)$} is epic, then $\mathcal{B}_{\zeta}\subseteq\cg T$.
\end{proposition}
\begin{proof} The proof can be induced from the implication (2)$\Rightarrow$(1) in \cite[Theorem 3.7]{Breaz}.
\end{proof}
Let $\mathcal{A}$ be an abelian category. Let $\mathcal{U}$ be a non-empty set (class) of objects in a
subcategory $\mathcal{C}\subseteq \mathcal{A}$. Recall that an object $A$ in $\mathcal{C}$ is said to be generated by $\mathcal{U}$ or $\mathcal{U}$-generated if, for every pair of distinct morphisms $f, g : A \rightarrow B$ in $\mathcal{A}$, there is a morphism $h : U \rightarrow A$ with $U\in \mathcal{U}$ and $f h \neq g h$. In this case, $\mathcal{U}$  is called a set (class) of
generators for  $\mathcal{C}$. Dually, one can define the notions  of cogenerated and cogenerators. Assume that the coproduct and product of $U$ in $\mathcal{A}$ exist. $\gen_{\mathcal{A}} U$ denotes the objects class of  $\mathcal{A}$ which each object is an image of a coproduct of $U$ in $\mathcal{A}$. $\cg_{\mathcal{A}} U$ denotes the  objects class of  $\mathcal{A}$ which each object can be embedded into a  product of $U$ in $\mathcal{A}$.

For an $R$-module $L$ of $\Mod R$, the submodule
$$\Tr_{\mathcal{U}}(L)=\sum\{\Im f~|~f\in\Hom_{R}(U, L), ~U\in \mathcal{U}\}\subseteq L$$
is called the trace of $\mathcal{U}$ in $L$, and
$$\rej_{\mathcal{U}}(L)=\bigcap\{\Ker f~|~f\in\Hom_{R}(L, U), ~U\in \mathcal{U}\}\subseteq L$$
is called the reject of $\mathcal{U}$ in $L$.
It is easy to see that $\Tr_{\mathcal{U}}(L)$ is  the maximal submodule of $L$ generated by $\mathcal{U}$ and $\rej_{\mathcal{U}}(L)$ is the minimal submodule of $L$ such that $L/\rej_{\mathcal{U}}(L)$ is cogenerated by $\mathcal{U}$. If an $R$-module $L$ is generated by $\mathcal{U}$, then $\Tr_{\mathcal{U}}(L)=L$.\par

If $\mathcal{U}=\{U\}$, then $$\Tr_{U}(L)=\{\sum_{i=1}^{k}f_{i}(u_{i})~|~u_{i}\in U, f_{i}\in\Hom_{R}(U, L), k\in\mathbb{N}\}.$$ It is well-known that $\Tr_{U}(-)$ defines a functor
from $\Mod R$ into $\Mod R$. For any $R$-homomorphism $\zeta: L\rightarrow X$, $\Tr_{U}(\zeta):\Tr_{U}(L)\rightarrow \Tr_{U}(X) $ is defined by $\Tr_{U}(\zeta)(l)=\zeta(l)=\sum_{i=1}^{k}\zeta f_{i}(u_{i})$ for any $l=\sum_{i=1}^{k}f_{i}(u_{i})\in \Tr_{U}(L)$. $L$ is (co)generated by $\mathcal{U}$ if and only if $\Tr_{U}(L)=L$ ($\rej_{\mathcal{U}}(L)=0$). In particular, $\Tr_{\mathcal{U}}(\Tr_{\mathcal{U}}(L))=\Tr_{\mathcal{U}}(L)$ and $\rej_{\mathcal{U}}(L/\rej_{\mathcal{U}}(L))=0$. From \cite{Anderson}, if $I$ is a right ideal of $R$ and $M$ is an $R$-module, then $\Tr_{R/I}(M)=\rej_{I}(M)$. For more consequences about trace and reject, we refer to \cite{Anderson}  and \cite{Wisbauer}.\par
 Recall that a cocomplete abelian category $\mathcal{A}$ is called a Grothendieck category if direct limits are exact in $\mathcal{A}$ and $\mathcal{A}$ has a generator. Let $M$ be an $R$-module. Denote by $\sigma[M]$ the the full subcategory
of $\Mod R$ whose objects are submodules of $M$-generated modules. Recall from \cite{Wisbauer} that $\sigma[M]$ is a Grothendieck category with enough injective objects, see \cite[17.8]{Wisbauer}. For any family $\{T_{\lambda}\}_{\lambda\in\Lambda}$ of modules in $\sigma[M]$,  the product in $\sigma[M]$ exists, see \cite[15.1(6)]{Wisbauer}. If $W$ is an injective module  in $\Mod R$, $\Tr_{M}(W)$ is an injective module in $\sigma[M]$. Moreover, if $W$ is a cogenerator for $\Mod R$, $\Tr_{M}(W)$ is a cogenerator for $\sigma[M]$. If the $R$-module $M$ is finitely generated as a module over $\End_{R}(M)$, then $\sigma[M]=\Mod (R/\ann_{R}(M))$, where $\ann_{R}(M)=\{r\in R~|~Mr=0\}$.
\section{Main result}
 Recall that for an object $M$ of $\Mod R$, the perpendicular category $^{\perp}M$ of  $M$ in $\Mod R$ is defined as follows
$$^{\perp}M=\{X\in\Mod R~|~\Ext^{1}_{R}(X,M)=0\}.$$
Let $\zeta:Q_{0}\rightarrow Q_{1}$ be an $R$-homomorphism in $\Mod R$. It has been shown that if $Q_{0}$ is injective and $T=\Ker \zeta$, then $\mathcal{B}_{\zeta}\subseteq {^{\perp}T}$, see \cite[Lemma 2.3 (4)]{Breaz}.
\begin{lemma}\label{lemma3-4}
Let $M$ be an $R$-module.  If for any  set $\Lambda$, $T^{\Lambda}\in\sigma[M]$ where $T^{\Lambda}$ is the product of $T$ in $\Mod R$, then the product $\prod_{\lambda\in\Lambda}^{M}T_{\lambda}$ of $T$ in $\sigma[M]$ is just the product of $T$ in $\Mod R$, that is $\prod_{\lambda\in\Lambda}^{M}T_{\lambda}=T^{\Lambda}$.
\end{lemma}
\begin{proof} Assume that $\mathcal{M}_{f}=\{U\subset M^{\mathbb{N}}~| U~\text{finitely generated}\}$. From \cite[15.1(4)]{Wisbauer}, $U_{f}=\bigoplus\{U~|~U\in\mathcal{M}_{f}\}$ is a generator of $\sigma[M]$. Since $T^{\Lambda}\in\sigma[M]$, $T^{\Lambda}$ is generated by $U_{f}$. From \cite[15.1(6)]{Wisbauer}, we have that  $\prod_{\lambda\in\Lambda}^{M}T_{\lambda}=\Tr_{U_{f}}(T^{\Lambda})= T^{\Lambda}$.
\end{proof}

Let $\mathcal{A}$ be an abelian category. For any integer $i\geq1$, $\Ext^{i}_{\mathcal{A}}(A,B)$ denotes the extension group, which
is defined via its Yoneda description as certain equivalent classes of $i$-fold extensions. Assume that the product of  $T$ in $\mathcal{A}$ exists. Recall that $T$ is said to be partial cotilting  in an abelian category $\mathcal{A}$ if $\cg_{\mathcal{A}} T\subseteq \Ker\Ext^{1}_{\mathcal{A}}(-,T)$ in $\mathcal{A}$. Moreover, if $\cg_{\mathcal{A}} T= \Ker\Ext^{1}_{\mathcal{A}}(-,T)$  in $\mathcal{A}$, then $T$ is called cotilting.
\begin{proposition}\label{prop3-1}
Let $I$ be a right ideal of a ring $R$ and $\zeta:Q_{0}\rightarrow Q_{1}$ is an $R$-homomorphism between  injective modules with $T=\Ker \zeta$. If $T$ is a  cotilting object in $\sigma[R/I]$ with the injective coresolving
\begin{equation}\label{eq3-0}
0\longrightarrow T\longrightarrow   \Tr_{R/I}(Q_{0})\xrightarrow{\Tr_{R/I}(\zeta)}\Tr_{R/I}(Q_{1})\longrightarrow  0
\end{equation}
 in $\sigma[R/I]$ and $\Ext^{1}_{\sigma[R/I]}(R/I,T)=0$, then
 $R/I\in\mathcal{B}_{\zeta}$, $\cg T\subseteq\sigma[R/I]$ and $\cg T\subseteq\mathcal{B}_{\zeta}$.
\end{proposition}
\begin{proof}
 First, we prove that $R/I\in\mathcal{B}_{\zeta}$. Applying the functor $\Hom_{\sigma[R/I]}(R/I,-)$  to the short exact sequence (\ref{eq3-0}),
we have the following exact sequence
$$\Hom_{\sigma[R/I]}(R/I,\Tr_{R/I}(Q_{0}))\rightarrow\Hom_{\sigma[R/I]}(R/I,\Tr_{R/I}(Q_{1}))\rightarrow\Ext^{1}_{\sigma[R/I]}(R/I,T)$$
From \cite[Exercises 8 (7)]{Anderson},  we have isomorphisms
$\Hom_{R}(R/I,\Tr_{R/I}(Q_{i}))\cong \Hom_{R}(R/I,Q_{i})$, for $i=1,2$. Since $\sigma[R/I]$ is a full subcategory of $\Mod R$,  we know that $ \Hom_{\sigma[R/I]}(R/I,\Tr_{R/I}(Q_{0}))=\Hom_{R}(R/I,\Tr_{R/I}(Q_{i}))$, for $i=1,2$.
Note that $\Ext^{1}_{\sigma[R/I]}(R/I,T)=0$. Thus, we have that $R/I\in\mathcal{B}_{\zeta}$.\par
Assume that $T$ is not zero.  It remains to show that $T^{\mu}$ is generated by $R/I$  and $T^{\mu}\in \mathcal{B}_{\zeta}$ for all sets $\mu$ since $\sigma[R/I]$ and $\mathcal{B}_{\zeta}$ are closed under submodules.
We divided the proof into several steps.
\smallskip

\textbf{Step 1}. We prove that $\Tr_{R/I}(T^{\mu})$ is a nonzero submodule of $T^{\mu}$.
\smallskip

 If $\Hom_{R}(R/I,T )=0$, then $R/I\in\Ker\Hom_{\sigma[R/I]}(-,T)$. Note that  $R/I\in \mathcal{B}_{\zeta}\subseteq {^{\perp}T}$. Then $R/I\in \Ker\Ext^{1}_{\sigma[R/I]}(-,T)$ since $\Ext^{1}_{\sigma[R/I]}(R/I,T)\subseteq \Ext^{1}_{R}(R/I,T)$. Since $T$ is a cotilting object in $\sigma[R/I]$, $\cg_{\sigma[R/I]} T=\Ker\Ext^{1}_{\sigma[R/I]}(-,T)$.  Thus, $R/I$ is cogenerated by $T$ in $\sigma[R/I]$. Then there exist a monomorphism $R/I\rightarrow \prod_{\theta\in\Theta}^{R/I}T_{\theta}$, where $\prod_{\theta\in\Theta}^{R/I}T_{\theta}$ is the product of $T$ in $\sigma[R/I]$. By \cite[15.1(6)]{Wisbauer}, there exists a monomorphism $\prod_{\theta\in\Theta}^{R/I}T_{\theta}\rightarrow T^{\Theta}$, where $T^{\Theta}$ is the product of $T$ in $\Mod R$. Thus, there is a monomorphism $R/I\rightarrow T^{\Theta}$. It implies that $\Hom_{R}(R/I,R/I)=0$. Thus, $R/I=0$ and so $\Tr_{R/I}(Q_{0})=0$. Form the exact sequence (\ref{eq3-0}), we know that $T=0$. It is a contradiction. Hence, $\Hom_{R}(R/I,T )\neq0$ and so $\Hom_{R}(R/I,T^{\mu} )\neq0$. Therefore, $\Tr_{R/I}(T^{\mu})$ is a nonzero module.\par
\smallskip

\textbf{Step 2}. We prove that $T^{\mu}$ is generated by $R/I$.
\smallskip

From \cite[Exercises 8 (7)]{Anderson}, we have an isomorphism $$\Hom_{R}(R/I, \Tr_{R/I}(T^{\mu}))\cong \Hom_{R}(R/I, T^{\mu}).$$ Since $\Tr_{R/I}(T^{\mu})$ is generated by $R/I$, there exists a nonzero epimorphism of $R$-modules  $f:{R/I}^{(\lambda)} \rightarrow\Tr_{R/I}(T^{\mu})$ where $\lambda$ is a set. For the set $\lambda$, there exists isomorphism $$\Hom_{R}({R/I}^{(\lambda)}, \Tr_{R/I}(T^{\mu}))\cong \Hom_{R}({R/I}^{(\lambda)}, T^{\mu}).$$ Then there exists an epimorphism $f':{R/I}^{(\lambda)} \rightarrow T^{\mu}$. Thus, $T^{\mu}$ is generated by $R/I$. Thus, $\cg T\subseteq \sigma[R/I]$ since $\sigma[R/I]$ is closed under submodules.
\smallskip

\textbf{Step 3.} We claim that $\Hom_{R}(T^{\mu},M/\Tr_{R/I}(M))=0$ for any $M\in\Mod R$.
\smallskip

 We set $h\in\Hom_{R}(T^{\mu},M/\Tr_{R/I}(M))$. For the epimorphism $f':{R/I}^{(\lambda)} \rightarrow T^{\mu}$, let $\varepsilon_{\lambda_{i}}:(R/I)_{\lambda_{i}}\rightarrow {R/I}^{(\lambda)}$ be the $\lambda_{i}$-th  canonical embedding with $\lambda_{i}\in\lambda$. Then for any $t\in T^{\mu}$, there exists a family $\{\overline{r_{i}}\}^{n}_{i=1}\in R/I$ such that  $\sum_{i}^{n}f'\varepsilon_{\lambda_{i}}(\overline{r_{i}})=t$. Hence, $h(t)=\sum_{i}^{n}hf'\varepsilon_{\lambda_{i}}(\overline{r_{i}})\in \Tr_{R/I}(M/\Tr_{R/I}(M))$. Note that $\Tr_{R/I}(M/\Tr_{R/I}(M))=\rej_{I}(M/\rej_{I}(M))=0$. Thus, we know that  $\Hom_{R}(T^{\mu},M/\Tr_{R/I}(M))=0$.
 \smallskip

\textbf{Step 4.} We prove that $T^{\mu}\in\mathcal{B}_{\zeta}$.
\smallskip

Since $T^{\mu}$ is generated by $R/I$, $T^{\mu}\in\sigma[R/I]$.  By Lemma \ref{lemma3-4},   $T^{\mu}=\prod_{\kappa\in\mu}^{R/I}T_{\kappa}$ in $\sigma[R/I]$.   Applying the functor $\Hom_{\sigma[R/I]}(T^{\mu},-)$ to the short exact sequence (\ref{eq3-0}), we have the following exact sequence
 $$\Hom_{\sigma[R/I]}(T^{\mu},\Tr_{R/I}(Q_{0}))\rightarrow \Hom_{\sigma[R/I]}(T^{\mu},\Tr_{R/I}(Q_{1}))\rightarrow\Ext^{1}_{\sigma[R/I]}(T^{\mu},T)$$
Since $T$ is a partial cotilting object in $\sigma[R/I]$ and  $\sigma[R/I]$ is a full subcategory of $\Mod R$, we have $\Ext^{1}_{\sigma[R/I]}(T^{\mu},T)=\Ext^{1}_{\sigma[R/I]}(\prod_{\kappa\in\mu}^{R/I}T_{\kappa},T)=0$ and so, $$\Hom_{R}(T^{\mu},\Tr_{R/I}(\zeta)):\Hom_{R}(T^{\mu},\Tr_{R/I}(Q_{0}))\rightarrow \Hom_{R}(T^{\mu},\Tr_{R/I}(Q_{1}))$$ is surjective. Applying $\Hom_{R}(T^{\mu},-)$ to the short exact sequence (\ref{eq3-0}), we have the following exact sequence
\begin{equation}\label{eq3-1}
\Hom_{R}(T^{\mu},\Tr_{R/I}(Q_{0}))\rightarrow \Hom_{R}(T^{\mu},\Tr_{R/I}(Q_{1}))\rightarrow\Ext^{1}_{R}(T^{\mu},T)\rightarrow\Ext^{1}_{R}(T^{\mu},\Tr_{R/I}(Q_{0})).
\end{equation}
Now, we shall prove that $\Ext^{1}_{R}(T^{\mu},\Tr_{R/I}(Q_{0}))=0$.
Applying $\Hom_{R}(T,-)$ to the short exact sequence $$0\rightarrow \Tr_{R/I}(Q_{0})\rightarrow  Q_{0}\rightarrow Q_{0}/\Tr_{R/I}(Q_{0})\rightarrow 0,$$ we have the following exact sequence
\begin{equation}\label{eq3-2}
\Hom_{R}(T^{\mu},Q_{0}/\Tr_{R/I}(Q_{0}))\rightarrow\Ext^{1}_{R}(T^{\mu},\Tr_{R/I}(Q_{0}))\rightarrow\Ext^{1}_{R}(T^{\mu},Q_{0}).
\end{equation}
  From the injectivity of $Q_{0}$ in $\Mod R$ and Step 3, we have that $\Ext^{1}_{R}(T^{\mu},\Tr_{R/I}(Q_{0}))=0$.\par
 From the exact sequence (\ref{eq3-1}), we obtain that $\Ext^{1}_{R}(T^{\mu},T)=0$. That is, $T^{\mu}\in {^{\perp}T}$.\par
 Now, we have a short exact sequence $0\rightarrow K\rightarrow {R/I}^{(\lambda)}\xrightarrow{f'} T^{\mu}\rightarrow 0$ with ${R/I}^{(\lambda)}\in \mathcal{B}_{\zeta}$. Since $ \mathcal{B}_{\zeta}$ is closed under submodules, $K\in \mathcal{B}_{\zeta}$. Then we have the following exact commutative diagram
 $$\xymatrix{&0\ar[d]&0\ar[d]&0\ar[d]&\\
 0\ar[r]&\Hom_{R}(T^{\mu}, T)\ar[r]\ar[d]&\Hom_{R}(T^{\mu}, Q_{0})\ar[r]\ar[d]&\Hom_{R}(T^{\mu}, Q_{1})\ar[d]&\\
 0\ar[r]&\Hom_{R}({R/I}^{(\lambda)}, T)\ar[r]\ar[d]&\Hom_{R}({R/I}^{(\lambda)}, Q_{0})\ar[r]\ar[d]&\Hom_{R}({R/I}^{(\lambda)}, Q_{1})\ar[r]\ar[d]&0\\
   0\ar[r]&\Hom_{R}(K, T)\ar[d]\ar[r]&\Hom_{R}(K, Q_{0})\ar[r]\ar[d]&\Hom_{R}(K, Q_{1})\ar[r]\ar[d]&0\\                 &0&0&0& }$$
   Applying Snake Lemma, we know that the first row is a short exact sequence, and so $T^{\mu}\in\mathcal{B}_{\zeta}$.
\end{proof}
\begin{lemma}\label{lemma3-3}
Let $I$ be a right ideal of a ring $R$ and $\zeta:Q_{0}\rightarrow Q_{1}$ is an $R$-homomorphism between injective modules with $T=\Ker \zeta$. If $T$ is a cotilting object in $\sigma[R/I]$ such that $R/I\in {^{\perp}T}$, $\cg T\subseteq \sigma[R/I]$ and $\cg T\subseteq {^{\perp}T}$, then for any injective $R$-module $W$, there exists a short exact sequence in $\sigma[R/I]$
 $$0\rightarrow T_{1}\rightarrow T_{0}\rightarrow \Tr_{R/I}(W)\rightarrow 0,$$
where $T_{i}\in \Prod T$.
\end{lemma}
\begin{proof}
Assume that $W$ is an injective $R$-module. Since $\Tr_{R/I}(W)$ is generated by $R/I$, there is an epimorphism $f:{R/I}^{(\lambda)}\rightarrow\Tr_{R/I}(W)$ in $\sigma[R/I]$. By the assumption on $R/I$, ${R/I}^{(\lambda)}\in \Ker\Ext^{1}_{\sigma[R/I]}(-,T)$ since $\Ext^{1}_{\sigma[R/I]}({R/I}^{(\lambda)},T)\subseteq \Ext^{1}_{R}({R/I}^{(\lambda)},T)=\Ext^{1}_{R}({R/I},T)^{\lambda}=0$. As $T$ is a cotilting object in $\sigma[R/I]$, $\cg_{\sigma[R/I]}T=\Ker\Ext^{1}_{\sigma[R/I]}(-,T)$. Thus, there is a monomorphism ${R/I}^{(\lambda)}\rightarrow \prod_{\lambda\in\Lambda}^{R/I}T$ in $\sigma[R/I]$. Since $\cg T\subseteq \sigma[R/I]$, $\prod_{\lambda\in\Lambda}^{R/I}T=T^{\Lambda}$ by Lemma \ref{lemma3-4}. Then there is a monomorphism ${R/I}^{(\lambda)}\rightarrow T^{\Lambda}$ in $\sigma[R/I]$. By the injectivity of $\Tr_{R/I}(W)$ in $\sigma[R/I]$, there is a morphism $f':T^{\Lambda}\rightarrow \Tr_{R/I}(W)$ such that the following diagram is commutative.
$$\xymatrix{
  0\ar[r]&{R/I}^{(\lambda)} \ar[d]_{f} \ar[r]^{} & \ar[dl]^{f'} T^{\Lambda}      \\
  &\Tr_{R/I}(W) }$$
It implies that $f' $  is epic since $f$ is surjective. Hence, we have a short exact sequence in  $\sigma[R/I]$
 $$0\rightarrow K\rightarrow T^{\Lambda}\xrightarrow{f'} \Tr_{R/I}(W)\rightarrow0 $$
where $K=\Ker f'\in \cg T$. By \cite[Lemma 4.2.1]{Colby}, there exists a short exact sequence in $\sigma[R/I]$
$$0\rightarrow K\xrightarrow{\eta} T^{\kappa}\rightarrow Y\rightarrow0 $$
such that  $\Hom_{R}(\eta,T)$ is surjective. Applying $\Hom_{R}(-,T)$ to the above sequence, we have the following exact sequence
$$0\rightarrow \Ext_{R}^{1}(Y,T)\rightarrow \Ext_{R}^{1}(T^{\kappa},T)\rightarrow \Ext_{R}^{1}(K,T).$$
Since $T^{\kappa}\in {^{\perp}T}$, $\Ext_{R}^{1}(Y,T)=0$ and so $\Ext_{\sigma[R/I]}^{1}(Y,T)=0$. Then $Y\in \cg_{\sigma[R/I]}T$ since $\cg_{\sigma[R/I]}T=\Ker\Ext^{1}_{\sigma[R/I]}(-,T)$. Thus, $Y\in \cg T$. Consider the push-out diagram in $\sigma[R/I]$
\begin{equation}\label{eq3-4}
\begin{split}
                     \xymatrix{&0\ar[d]&0\ar[d]&&\\
0\ar[r]&K\ar[r]\ar[d]&T^{\Lambda}\ar[d]\ar[r]&\Tr_{R/I}(W)\ar@{=}[d]\ar[r]&0\\
  0\ar[r]&T^{\kappa}\ar[d]\ar[r]&U\ar[r]\ar[d]&\Tr_{R/I}(W)\ar[r]&0 \\
  &Y\ar@{=}[r]\ar[d]&Y\ar[d]&& \\
  &0&0&&           }
                    \end{split}
\end{equation}
  Since $T^{\Lambda}$ and $Y$ are in ${^{\perp}T}$, we have $U\in{^{\perp}T}$. Thus, $U\in \Ker\Ext^{1}_{\sigma[R/I]}(-,T)$. By the similar arguments to $Y$, we know that $U\in \cg T$. Applying \cite[Lemma 4.2.1]{Colby} again and the similar arguments to $K$, there exists a short exact sequence in $\sigma[R/I]$
  \begin{equation}\label{eq3-5}
  0\rightarrow U\rightarrow T^{\alpha}\rightarrow L\rightarrow0.
\end{equation}
  where $L\in{^{\perp}T}$. Applying $\Hom_{\sigma[R/I]}(L,-)$ to the middle row in the diagram (\ref{eq3-4}), we have the following exact sequence
  $$\Ext^{1}_{\sigma[R/I]}(L, T^{\kappa})\rightarrow \Ext^{1}_{\sigma[R/I]}(L, U)\rightarrow \Ext^{1}_{\sigma[R/I]}(L,\Tr_{R/I}(W)).$$
  Since $\Ext^{1}_{\sigma[R/I]}(L, T^{\kappa})\subseteq \Ext^{1}_{R}(L, T^{\kappa})$ and $L\in{^{\perp}T}$, we know that $\Ext^{1}_{\sigma[R/I]}(L, T^{\kappa})=0$. By the injectivity of $\Tr_{R/I}(W)$ in $\sigma[R/I]$ , we have $\Ext^{1}_{\sigma[R/I]}(L,\Tr_{R/I}(W))=0$. It yields that $\Ext^{1}_{\sigma[R/I]}(L, U)=0$. It means that each short exact sequence $0\rightarrow U\rightarrow N\rightarrow L\rightarrow 0$ in $\sigma[R/I]$  is split. Thus the sequence (\ref{eq3-5}) is split and so $U\in \Prod T$.  The short exact sequence
  $$0\rightarrow T^{\kappa}\rightarrow U\rightarrow \Tr_{R/I}(W)\rightarrow0$$
  is the desired sequence.
\end{proof}
\begin{proposition}\label{prop3-4}
Let $I$ be a right ideal of a ring $R$ and $\zeta:Q_{0}\rightarrow Q_{1}$ is an $R$-homomorphism between  injective modules with $T=\Ker \zeta$. If $T$ and $I$ satisfy the following conditions:
  \begin{enumerate}
    \item $T$ is a cotilting object in $\sigma[R/I]$ with the injective coresolving
\begin{equation*}
0\longrightarrow  T\longrightarrow   \Tr_{R/I}(Q_{0})\xrightarrow{\Tr_{R/I}(\zeta)}\Tr_{R/I}(Q_{1})\longrightarrow  0.
\end{equation*}
    \item $\Ext^{1}_{\sigma[R/I]}(R/I,T)=0$.
    \item  {\rm$I\in\Ker\Hom_{R}(-,\mathcal{B}_{\zeta})$}.
  \end{enumerate}
then $T$ is a cosliting module in $\Mod R$.
\end{proposition}
\begin{proof} By the definition of cosilting modules and Proposition \ref{prop3-1}, it is enough to show that $\mathcal{B}_{\zeta}\subseteq \cg T$. Assume that $I\in\Ker\Hom_{R}(-,\mathcal{B}_{\zeta})$ and $W$ is an injective cogenerator of $\Mod R$. For any $X\in\mathcal{B}_{\zeta}$, applying $\Hom_{R}(-,X)$ to the short exact sequence
$$0\rightarrow I\rightarrow R\rightarrow R/I\rightarrow 0,$$
we have the isomorphism $\Hom_{R}(R/I, X)\cong\Hom_{R}(R, X)$. Hence, $\Tr_{R/I}(X)=X$ and so, $X$ is generated by $R/I$. Then there is an epimorphism $f:{R/I}^{(\Lambda)}\rightarrow X$. Hence, for any $x\in X$, we can write it as $\sum_{i}^{n}f\varepsilon_{\lambda_{i}}(\overline{r_{i}})$ where $\varepsilon_{\lambda_{i}}:(R/I)_{\lambda_{i}}\rightarrow {R/I}^{(\Lambda)}$ is the $\lambda_{i}$-th  canonical embedding with $\lambda_{i}\in\Lambda$ and $\overline{r_{i}}\in R/I$. Hence, for any $g\in\Hom_{R}(X,W)$, $g(x)=\sum_{i}^{n}gf\varepsilon_{\lambda_{i}}(\overline{r_{i}})\in\Tr_{R/I}(W)$. Thus, $\Hom_{R}(X,W)= \Hom_{R}(X,\Tr_{R/I}(W))$. By Proposition \ref{prop3-1} and Lemma \ref{lemma3-3}, there is a short exact sequence
$$0\rightarrow T_{1}\rightarrow T_{0}\rightarrow \Tr_{R/I}(W)\rightarrow 0,$$
where $T_{i}\in \Prod T$. Applying $\Hom_{R}(X,-)$ to this sequence, we have the short exact sequence
$$0\rightarrow \Hom_{R}(X,T_{1})\rightarrow \Hom_{R}(X,T_{0})\rightarrow \Hom_{R}(X,\Tr_{R/I}(W))\rightarrow \Ext^{1}_{R}(X,T_{1}).$$
Since $\mathcal{B}_{\zeta}\subseteq {^{\perp}T}$, $\Ext^{1}_{R}(X,T_{1})=0$. Then, we have the short exact sequence
$$0\rightarrow \Hom_{R}(X,T_{1})\rightarrow \Hom_{R}(X,T_{0})\rightarrow \Hom_{R}(X,W)\rightarrow 0.$$
By Proposition \ref{prop2-2}, we have that $\mathcal{B}_{\zeta}\subseteq \cg T$.
\end{proof}
\begin{lemma}\label{lemma3-1}
Let $I$ be a right ideal of a ring $R$ and $\zeta:Q_{0}\rightarrow Q_{1}$ be an $R$-homomorphism with $T=\Ker \zeta$. If $R/I\in\mathcal{B}_{\zeta}$ and {\rm$\Hom_{R}(I,T)=0$}, then  there exists a short exact sequence in $\Mod R$
$$0\longrightarrow  T\longrightarrow  \Tr_{R/I}(Q_{0})\xrightarrow{\Tr_{R/I}(\zeta)}\Tr_{R/I}(Q_{1})\longrightarrow 0.$$
\end{lemma}
\begin{proof}
Since $R/I\in\mathcal{B}_{\zeta}$, then we have a short exact sequence
\begin{equation*}
  0\longrightarrow  \Hom_{R}(R/I,T)\xrightarrow{ \iota^{\ast}} \Hom_{R}(R/I,Q_{0})\xrightarrow{\zeta^{\ast}}\Hom_{R}(R/I,Q_{1})\longrightarrow  0
\end{equation*}
where $\iota:T\rightarrow Q_{0}$ is an inclusion map. Then we claim that we have the following exact sequence
\begin{equation*}
  0\longrightarrow \Tr_{R/I}(T)\xrightarrow{\Tr_{R/I}(\iota)}\Tr_{R/I}(Q_{0})\xrightarrow{\Tr_{R/I}(\zeta)}\Tr_{R/I}(Q_{1})
\longrightarrow 0
\end{equation*}

First, we prove that $\Tr_{R/I}(\zeta):\Tr_{R/I}(Q_{0})\rightarrow\Tr_{R/I}(Q_{1})$ is surjective. We set $z\in \Tr_{R/I}(Q_{1})$. Then, $z$ can be written as $\sum_{i=1}^{k}h_{i}(\overline{r_{i}})$,  where $h_{i}\in\Hom_{R}(R/I,Q_{1})$ and $\overline{r_{i}}\in R/I$. Note that $\zeta^{\ast}$ is surjective. That is, the induced map $\Hom_{R}(R/I,\zeta):\Hom_{R}(R/I,Q_{0})\rightarrow\Hom_{R}(R/I,Q_{1})$ is surjective. Then there exists a collection $\{g_{i}:R/I\rightarrow Q_{0}\}$ of $R$-homomorphisms such that $h_{i}=\zeta g_{i}$ for $1\leq i\leq k$. Hence, $z=\sum_{i=1}^{k}h_{i}(\overline{r_{i}})=\zeta(\sum_{i=1}^{k} g_{i}(\overline{r_{i}}))=\Tr_{R/I}(\zeta)(\sum_{i=1}^{k} g_{i}(\overline{r_{i}}))$, where $\sum_{i=1}^{k} g_{i}(\overline{r_{i}})\in\Tr_{R/I}(Q_{0})$. Hence, $\Tr_{R/I}(\zeta):\Tr_{R/I}(Q_{0})\rightarrow\Tr_{R/I}(Q_{1})$ is surjective.\par
Second, we prove that $\Tr_{R/I}(\iota):\Tr_{R/I}(T)\rightarrow\Tr_{R/I}(Q_{0})$ is injective. For any $x\in \Ker\Tr_{R/I}(\iota)$, we write  $x=\sum_{i=1}^{k}f_{i}(\overline{r_{i}})$. Then $\Tr_{R/I}(\iota)(x)= \iota(\sum_{i=1}^{k}f_{i}(\overline{r_{i}}))=0$. Since $\iota$ is injective, $x=\sum_{i=1}^{k}f_{i}(\overline{r_{i}})=0$. Hence, $\Tr_{R/I}(\iota):\Tr_{R/I}(T)\rightarrow\Tr_{R/I}(Q_{0})$ is injective.\par
Third, we prove that $\Ker\Tr_{R/I}(\zeta)=\Im\Tr_{R/I}(\iota)$. For any $x=\sum_{i=1}^{k}f_{i}(\overline{r_{i}})\in\Tr_{R/I}(T)$, we have that $\Tr_{R/I}(\zeta)\Tr_{R/I}(\iota)(\sum_{i=1}^{k}f_{i}(\overline{r_{i}}))=\sum_{i=1}^{k}\zeta\iota f_{i}(\overline{r_{i}})=0$ since $\zeta\iota=0$. Thus, $\Im\Tr_{R/I}(\iota)\subseteq\Ker\Tr_{R/I}(\zeta)$. On the other hand, for any $y=\sum_{i=1}^{k}g_{i}(\overline{r_{i}})\in\Ker\Tr_{R/I}(\zeta)$, $\Tr_{R/I}(\zeta)(\sum_{i=1}^{k}g_{i}(\overline{r_{i}}))=\zeta(\sum_{i=1}^{k}g_{i}(\overline{r_{i}}))=0$. Then, $y\in\Ker \zeta$. Note that $\Ker \zeta=\Im \iota$. Thus, there exists $x\in T$ such that $\iota(x)=y$. By the assumption, we know that $\Hom_{R}(\pi, T):\Hom_{R}(R/I, T)\rightarrow \Hom_{R}(R, T)$ is an isomorphism, where $\pi:R\rightarrow R/I$ is a projection. Then, $T$ is generated by $R/I$ and for $x\in T$, there exists an $R$-homomorphism $h:R/I\rightarrow T$ such that $\Hom_{R}(\pi, T)(h)(1)=h\pi(1)=x$ where $1$ is the unit of $R$. Thus, $y=\iota(h\pi(1))\in \Im\Tr_{R/I}(\iota)$ and $\Ker\Tr_{R/I}(\zeta)\subseteq\Im\Tr_{R/I}(\iota)$.\par
We get the desired sequence from the fact $\Tr_{R/I}(T)=T$ which has been proved above.
\end{proof}
\begin{lemma}\label{lemma3-6}Let $R$ be a ring and $M$ be an $R$-module. If $X$ and $\{Y_{\lambda}\}_{\lambda\in\Lambda}$ are $R$-modules in $\sigma[M]$, then there is a monomorphism {\rm$\Hom_{\sigma[M]}(X,\prod_{\lambda\in\Lambda}^{M}Y_{\lambda})\rightarrow \prod_{\lambda\in\Lambda}\Hom_{\sigma[M]}(X,Y_{\lambda})$}.
\end{lemma}
\begin{proof} By the construction of products in  $\sigma[M]$ (see \cite[15.1(6)]{Wisbauer}), we know that $\prod_{\lambda\in\Lambda}^{M}Y_{\lambda}$ together with the restrictions $\overline{\pi}_{\lambda}$ of the canonical projections $\pi_{\lambda}:\prod_{\lambda\in\Lambda}Y_{\lambda}\rightarrow Y_{\lambda} $ is the product in $\sigma[M]$. Then we can define two group homomorphisms $\phi:\Hom_{\sigma[M]}(X,\prod_{\lambda\in\Lambda}^{M}Y_{\lambda})\rightarrow \prod_{\lambda\in\Lambda}\Hom_{\sigma[M]}(X,Y_{\lambda})$ and $\Phi:\Hom_{R}(X,\prod_{\lambda\in\Lambda}Y_{\lambda})\rightarrow \prod_{\lambda\in\Lambda}\Hom_{R}(X,Y_{\lambda})$, which are given by $\phi(f)=\{f\overline{\pi}_{\lambda}\}_{\lambda\in\Lambda}$ and $\Phi(g)=\{g\pi_{\lambda}\}_{\lambda\in\Lambda}$ for any $f\in\Hom_{\sigma[M]}(X,\prod_{\lambda\in\Lambda}^{M}Y_{\lambda})$, $g\in\Hom_{R}(X,\prod_{\lambda\in\Lambda}Y_{\lambda})$. It is well-known that $\Phi$ is an isomorphism. Consider the following commutative diagram
 $$\xymatrix{0\ar[r]&\Hom_{\sigma[M]}(X,\prod_{\lambda\in\Lambda}^{M}Y_{\lambda})\ar[r]^{inc}\ar[d]^{\phi}  & \Hom_{R}(X,\prod_{\lambda\in\Lambda}Y_{\lambda})\ar[d]^{\Phi}\\
   0\ar[r]&\prod_{\lambda\in\Lambda}\Hom_{\sigma[M]}(X,Y_{\lambda})\ar[r]^{id}  & \prod_{\lambda\in\Lambda}\Hom_{R}(X,Y_{\lambda})       }$$
   where $inc$ is a inclusion map and $id$ is an identity. It implies that $\phi$ is injective.
\end{proof}
\begin{lemma}\label{lemma3-7}Let $R$ be a ring and $M$ be an $R$-module. If $T$ is an $R$-module such that $T^{\Lambda}\in\sigma[M]$ for any set $\Lambda$, then there is a monomorphism $\Ext^{1}_{\sigma[M]}(X,T^{\Lambda})\rightarrow \Ext^{1}_{\sigma[M]}(X,T)^{\Lambda}$ for any $X\in\sigma[M]$.
\end{lemma}
\begin{proof}
 Assume that $[\xi]\in\Ext^{1}_{R}(X,T^{\Lambda})$ where $\xi:0\rightarrow T^{\Lambda}\rightarrow N\rightarrow X\rightarrow0$. Let $\lambda\in\Lambda$ and $\pi_{\lambda}:T^{\Lambda}\rightarrow T_{\lambda}$ be the $\lambda$-th projection. Consider the push-out
  $$\xymatrix{\xi:& 0\ar[r]&T^{\Lambda}\ar[d]_{\pi_{\lambda}}\ar[r]&N \ar[d]\ar[r]&X\ar@{=}[d]\ar[r]&0\\
   (\pi_{\lambda})_{\sharp}(\xi):&0\ar[r]&T_{\lambda}\ar[r]&U \ar[r]&X\ar[r]&0  }$$
  $(\pi_{\lambda})_{\sharp}(\xi)$ denotes by the bottom row exact sequence of the above diagram.
  It is  well-known that the group homomorphism $\Phi:\Ext^{1}_{R}(X,T^{\Lambda})\rightarrow \Ext^{1}_{R}(X,T)^{\Lambda}$ given by $\Phi([\xi])=\{[(\pi_{\lambda})_{\sharp}(\xi)]\}_{\lambda\in\Lambda}$, where $[\xi]\in \Ext^{1}_{R}(X,T^{\Lambda})$, is an isomorphism. If $[\xi]\in\Ext^{1}_{\sigma[M]}(X,T^{\Lambda})$, then $(\pi_{\lambda})_{\sharp}(\xi)\in\Ext^{1}_{\sigma[M]}(X,T)$ by \cite[15.1(5)]{Wisbauer}.  Hence, we get the restriction of $\Phi$ on $\Ext^{1}_{\sigma[M]}(X,T^{\Lambda})$, denoted by $\Psi:\Ext^{1}_{\sigma[M]}(X,T^{\Lambda})\rightarrow \Ext^{1}_{\sigma[M]}(X,T)^{\Lambda}$. Then there is a commutative diagram
  $$\xymatrix{0\ar[r]&\Ext^{1}_{\sigma[M]}(X,T^{\Lambda})\ar[r]\ar[d]^{\Psi}  & \Ext^{1}_{R}(X,T^{\Lambda})\ar[d]^{\Phi}\\
   0\ar[r]&\Ext^{1}_{\sigma[M]}(X,T)^{\Lambda}\ar[r]  & \Ext^{1}_{R}(X,T)^{\Lambda}       }$$
   where all rows are exact. Therefore, we know that $\Psi$  is injective.
\end{proof}
\begin{theorem}\label{Th1}
Let $R$ be a ring and $\zeta:Q_{0}\rightarrow Q_{1}$ is an $R$-homomorphism between injective modules with $T=\Ker \zeta$. Then $T$ is a cosliting module in $\Mod R$ if and only if there exists a right ideal $I$ of  $R$ such that  $T$ and $I$ satisfy the following conditions:
  \begin{enumerate}
    \item $T$ is a  cotilting object in $\sigma[R/I]$ with the injective coresolving
\begin{equation*}
0\longrightarrow  T\longrightarrow   \Tr_{R/I}(Q_{0})\xrightarrow{\Tr_{R/I}(\zeta)}\Tr_{R/I}(Q_{1})\longrightarrow  0.
\end{equation*}
    \item $\Ext^{1}_{\sigma[R/I]}(R/I,T)=0$.
    \smallskip

    \item  {\rm$I\in\Ker\Hom_{R}(-,\mathcal{B}_{\zeta})$}.
  \end{enumerate}
\end{theorem}
\begin{proof}
By Proposition \ref{prop3-4}, it is enough to show that the necessity. Assume that $T$ is a cosilting module in $\Mod R$. By Proposition \ref{lemma2-1}, $T$ induces a torsion pair ($\Ker\Hom_{R}(-,T)$, $\cg T$) in $\Mod R$. Let $I$ be the torsion part of $R$ with respect to the torsion pair ($\Ker\Hom_{R}(-,T)$, $\cg T$). We check that this $I$ is just the desired right ideal of $R$. It is easy to see that $\Hom_{R}(I,\mathcal{B}_{\zeta})=0$. Then for any $X\in\cg T$, we have an isomorphism $\Hom_{R}(R/I,X)\cong \Hom_{R}(R,X)$. Thus, $X$ is generated by $R/I$ and so $\cg T\subseteq \sigma[R/I]$. Note that  $R/I\in \cg T=\mathcal{B}_{\zeta}\subseteq {^{\perp}T}$. Since $\Ext^{1}_{\sigma[R/I]}(R/I,T)\subseteq \Ext^{1}_{R}(R/I,T)=0$, we have $\Ext^{1}_{\sigma[R/I]}(R/I,T)=0$.\par
Now, it is enough to check (1). Since $T\in\mathcal{B}_{\zeta}$, we know that $\Hom_{R}(I,T)=0$. By Lemma \ref{lemma3-1},  there exists a short exact sequence in $\sigma[R/I]$
$$0\longrightarrow  T\longrightarrow  \Tr_{R/I}(Q_{0})\xrightarrow{\Tr_{R/I}(\zeta)}\Tr_{R/I}(Q_{1})\longrightarrow 0.$$
Then the injective dimension of $T$ in $\sigma[R/I]$ is at most one.\par
 Next, we prove that $\Ext^{1}_{\sigma[R/I]}(T^{\lambda}, T)=0$. Note that $T\in\mathcal{B}_{\zeta}\subseteq {^{\perp}T}$ and $\mathcal{B}_{\zeta}$ is closed under product. Thus for any set $\lambda$, $T^{\lambda}\in{^{\perp}T}$. Moreover, since $\cg T\subseteq \sigma[R/I]$, $T^{\lambda}\in\sigma[R/I]$  and so $\Ext^{1}_{\sigma[R/I]}(T^{\lambda}, T)\subseteq \Ext^{1}_{R}(T^{\lambda}, T)=0$. It means that $\Ext^{1}_{\sigma[R/I]}(T^{\lambda}, T)=0$.\par
 From \cite[Theorem 3.7]{Breaz}, there exists a short exact sequence in $\sigma[R/I]$
 \begin{equation}\label{eq3-6}
 0\rightarrow T_{1}\rightarrow T_{0}\rightarrow\Tr_{R/I}(E) \rightarrow 0
\end{equation}
 where $T_{i}\in\Prod T$, $E$ is an injective cogenerator of $\Mod R$. By Lemma \ref{lemma3-4}, we know that $T_{i}\in\Prod_{\sigma[R/I]}T$, where $\Prod_{\sigma[R/I]}T$ consists of all direct summands of product of $T$ in $\sigma[R/I]$.\par

 We claim that $\Ker\Hom_{\sigma[R/I]}(-,T)\bigcap \Ker\Ext^{1}_{\sigma[R/I]}(-,T)=0$.\par For any $X\in\Ker\Hom_{\sigma[R/I]}(-,T)\bigcap \Ker\Ext^{1}_{\sigma[R/I]}(-,T)$, applying $\Hom_{\sigma[R/I]}(X,-)$ to the sequence (\ref{eq3-6}), we have the following exact sequence
 $$\Hom_{\sigma[R/I]}(X,T_{0})\rightarrow \Hom_{\sigma[R/I]}(X,\Tr_{R/I}(E))\rightarrow\Ext^{1}_{\sigma[R/I]}(X,T_{1}).$$
Since $\Hom_{\sigma[R/I]}(X,T^{\lambda})=\Hom_{R}(X,T)^{\lambda}=\Hom_{\sigma[R/I]}(X,T)^{\lambda}$, we have that $\Hom_{\sigma[R/I]}(X,T_{0})=0$. By Lemma \ref{lemma3-7} and $X\in\Ker\Ext^{1}_{\sigma[R/I]}(-,T)$, $\Ext^{1}_{\sigma[R/I]}(X,T_{1})=0$. Hence, we have that  $\Hom_{\sigma[R/I]}(X,\Tr_{R/I}(E))=0$. Since $\Tr_{R/I}(E)$ is an injective cogenerator, there is a monomorphism $X\rightarrow \prod_{\alpha\in J}^{R/I}\Tr_{R/I}(E)$. Then, by Lemma \ref{lemma3-6}, we have the monomorphism $g$ which is the composition of the following monomorphisms $$\Hom_{\sigma[R/I]}(X,X)\rightarrow \Hom_{\sigma[R/I]}(X,\Pi_{\alpha\in J}^{R/I}\Tr_{R/I}(E)_{\alpha})\rightarrow \Pi_{\alpha\in J}\Hom_{\sigma[R/I]}(X,\Tr_{R/I}(E)_{\alpha}).$$
Then we obtain that $\End_{R}(X)=0$ and so, $X=0$.\par
Now, we check that $\cg_{\sigma[R/I]}T=\Ker\Ext^{1}_{\sigma[R/I]}(-,T)$. For any $X\in \cg_{\sigma[R/I]}T$, there is a monomorphism $f:X\rightarrow \prod_{\alpha\in J}^{R/I}T_{\alpha}$. By Lemma \ref{lemma3-4}, we have that $\prod_{\alpha\in J}T_{\alpha}=T^{J}$. Hence, $X\in \cg T=\mathcal{B}_{\zeta}\subseteq {^{\perp}T}$. It implies that $X\in\Ker\Ext^{1}_{\sigma[R/I]}(-,T)$ and hence, $\cg_{\sigma[R/I]}T\subseteq\Ker\Ext^{1}_{\sigma[R/I]}(-,T)$. On the other hand, for any $X\in\Ker\Ext^{1}_{\sigma[R/I]}(-,T)$, there is a short exact sequence in $\sigma[R/I]$
$$0\rightarrow \rej_{T}(X)\rightarrow X\rightarrow X/\rej_{T}(X)\rightarrow 0.$$
Applying $\Hom_{\sigma[R/I]}(-,T)$ to this sequence, we have a long exact sequence
\begin{align*}
    0\rightarrow&\Hom_{\sigma[R/I]}(X/\rej_{T}(X),T) \xlongrightarrow{\cong}\Hom_{\sigma[R/I]}(X,T) \rightarrow\Hom_{\sigma[R/I]}(\rej_{T}(X),T)\\
   \rightarrow&\Ext^{1}_{\sigma[R/I]}(X/\rej_{T}(X),T)\rightarrow\Ext^{1}_{\sigma[R/I]}(X,T)
 \rightarrow\Ext^{1}_{\sigma[R/I]}(\rej_{T}(X),T)\\
 \rightarrow& \Ext^{2}_{\sigma[R/I]}(X/\rej_{T}(X),T)\rightarrow\cdots.
  \end{align*}
By the formula in \cite[Exercises 8 (7)]{Anderson} and $\sigma[R/I]$ is a full subcategory, we have the isomorphism $\Hom_{\sigma[R/I]}(X/\rej_{T}(X),T) \cong\Hom_{\sigma[R/I]}(X,T)$. Since $X/\rej_{T}(X)$ is cogenerated by $T$, from Lemma \ref{lemma3-4}, $X/\rej_{T}(X)\in\cg_{\sigma[R/I]}T$. Then, $\Ext^{1}_{\sigma[R/I]}(X/\rej_{T}(X),T)=0$. It yields that $\Hom_{\sigma[R/I]}(\rej_{T}(X),T)=0$. Since $X\in\Ker\Ext^{1}_{\sigma[R/I]}(-,T)$ and the injective dimension of $T$ in $\sigma[R/I]$ is at most one, we know that $\Ext^{1}_{\sigma[R/I]}(X,T)=0=\Ext^{2}_{\sigma[R/I]}(X/\rej_{T}(X),T)$. Thus, $\Ext^{1}_{\sigma[R/I]}(\rej_{T}(X),T)=0$. This means $\rej_{T}(X)\in\Ker\Hom_{\sigma[R/I]}(-,T)\bigcap \Ker\Ext^{1}_{\sigma[R/I]}(-,T)$. Hence, $\rej_{T}(X)=0$ and so, $X\cong X/\rej_{T}(X)\in\cg_{\sigma[R/I]}T$. Therefore, $\Ker\Ext^{1}_{\sigma[R/I]}(-,T)\subseteq\cg_{\sigma[R/I]}T$. This completes the proof.
\end{proof}
\begin{corollary}
Let $R$ be a ring and $\zeta:Q_{0}\rightarrow Q_{1}$ is an $R$-homomorphism between injective modules with $T=\Ker \zeta$. Assume that $T$ is a cosliting module in $\Mod R$ and $I$ is the torsion part of $R$ with respect to the torsion pair {\rm($\Ker\Hom_{R}(-,T)$, $\cg T$)}. If one of following conditions hold, \begin{enumerate}
  \item $R$ is commutative
  \item $R/I$  is finitely generated as a module over $\End_{R}(R/I)$
\end{enumerate}
then $T$ is a  cotilting module in $\Mod R/J$, where $J=\ann_{A}(R/I)$.
\end{corollary}
\begin{proof} It is a direct consequence of Theorem \ref{Th1} and \cite[15.4]{Wisbauer}.
\end{proof}
\begin{example}\label{ex1} Let $A$ be the $k$-algebra over a filed $k$ given by the bound quiver
$$\xymatrix@!0{3\ar[r]^{\alpha}&2\ar[r]^{\beta}&1}$$
with relation $\alpha\beta=0$. The Auslander-Reiten quiver of $\mod A$ can be drawn as following.
$$\xymatrix@!0{&*+<0.5mm>{\text{$\begin{smallmatrix}2\\1\end{smallmatrix}$}}\ar[dr]^{b}&&*+<0.5mm>{\text{$\begin{smallmatrix}3\\2\end{smallmatrix}$}}\ar[dr]&\\
*+<0.5mm>{\text{$\begin{smallmatrix}1\end{smallmatrix}$}}\ar[ur]^{a}\ar@{..}[rr]&&*+<0.5mm>{\text{$\begin{smallmatrix}2\end{smallmatrix}$}}\ar[ur]^{c}\ar@{..}[rr]&&*+<0.5mm>{\text{$\begin{smallmatrix}3\end{smallmatrix}$}}}$$
Then $\gl~A=2$. We consider $T=\begin{smallmatrix}1\end{smallmatrix}\oplus \begin{smallmatrix}2\\1\end{smallmatrix}$. Then there is an injective coresolving
$$\xymatrix{
                \text{$\begin{smallmatrix}0\end{smallmatrix}$}\ar[r]& \text{$\begin{smallmatrix}1\end{smallmatrix}$}\oplus \text{$\begin{smallmatrix}2\\1\end{smallmatrix}$} \ar[r]^-{\text{$\begin{bmatrix}\begin{smallmatrix}a\\1\end{smallmatrix}\end{bmatrix}$}}& \text{$\begin{smallmatrix}2\\1\end{smallmatrix}$}\oplus \text{$\begin{smallmatrix}2\\1\end{smallmatrix}$}\ar[r]^-{\zeta=\text{$\begin{bmatrix}\begin{smallmatrix}cb\\0\end{smallmatrix}\end{bmatrix}$}}&      \text{$\begin{smallmatrix}3\\2\end{smallmatrix}$}    }$$
It is easy to check that $\mathcal{B}_{\zeta}=\cg T=\Add T$. Hence, $T$ is a cosilting $A$-module. Set  $I=\begin{smallmatrix}3\\2\end{smallmatrix}$. Then $A/I\cong T$ in $\Mod A$  and $\ann_{A}(A/I)=Ae_{3}A$. Clearly, $T$ is a cotilting $A/\ann_{A}(A/I)$-module with the injective coresolving
$$\xymatrix{
                \text{$\begin{smallmatrix}0\end{smallmatrix}$}\ar[r]& \text{$\begin{smallmatrix}1\end{smallmatrix}$}\oplus \text{$\begin{smallmatrix}2\\1\end{smallmatrix}$} \ar[r]& \text{$\begin{smallmatrix}2\\1\end{smallmatrix}$}\oplus \text{$\begin{smallmatrix}2\\1\end{smallmatrix}$}\ar[r]&      \text{$\begin{smallmatrix}2\end{smallmatrix}$}\ar[r]&\text{$\begin{smallmatrix}0\end{smallmatrix}$}   }.$$
                where $\Tr_{A/I}(\begin{smallmatrix}3\\2\end{smallmatrix})=\begin{smallmatrix}2\end{smallmatrix}$.
\end{example}


\begin{thebibliography}{100}
\bibitem{Adachi}Adachi T., Iyama O. , Reiten I., $\tau$-Tilting theory. Compos. Math. 2014, 150(3), 415-452.
\bibitem{Angeleri} Angeleri-H\"{u}gel, L., Marks, F., Vit\'{o}ria, J., Silting modules. Int. Math. Res. Not. 2016, (4), 1251-1284.

\bibitem{Anderson}Anderson, F.W., Fuller, K.R., Rings and Categories of modules. Springer-Verlag, New York, 1973.

\bibitem{Breaz}Breaz S. and Pop F., Cosilting Modules. Algebr. Represent. Theory. 2017, 20:1305-1321.
\bibitem{Colby}Colby, R.R., Fuller, K.R., Equivalence and duality for module categories. With tilting and cotilting for rings. Cambridge University Press, Cambridge (2004).
\bibitem{Wisbauer}Wisbauer, R., Foundations of module and ring theory, Algebra, Logic and applications. 3. Gordon and Breach Science Publishers (1991).
\bibitem{Keller}    Keller B. and Vossieck D., Aisles in derived categories, Bull. Soc. Math. Belg. Ser. A, 1988, 40, 239-253.
\bibitem{Koenig}    Koenig, S., Yang, D.: Silting objects, simple-minded collections, $t$-structures and co-$t$-structures for
finite-dimensional algebras. Doc. Math. 2014, 19, 403-438.

\bibitem{Zhang}    Zhang P., Wei J., Cosilting complexes and AIR-cotilting modules. J. Algebra, 2017, 491, 1-31.
\bibitem{Zhang02}    Zhang P. , Wei J., Quasi-cotilting modules and torsion-free classes. J. Algebra Appl., 2016.
\end{thebibliography}
\end{document}